\documentclass[11pt,twoside,reqno]{article}
\usepackage{mathrsfs,subfigure,latexsym,array,multirow}
\usepackage{amsmath,amssymb,amsthm,datetime,graphicx}
\usepackage{color,comment,subeqnarray}
\usepackage[dvipsnames]{xcolor}
\title{ $P_1$-nonconforming  divergence-free finite element method 
on square meshes for Stokes equations 
}
\author{
Chunjae Park
\thanks{Department of Mathematics,
Konkuk University, Seoul 05029, Korea, 
\hspace{1mm} Emails: cjpark@konkuk.ac.kr}}

\begin{document}

\newtheorem{theorem}{Theorem}[section]
\newtheorem{remark}[theorem]{Remark}
\newtheorem{lemma}[theorem]{Lemma}
\newtheorem{proposition}[theorem]{Proposition}
\newtheorem{definition}[theorem]{Definition}
\newtheorem{assumption}{Assumption}[section] 

\def\disp{\displaystyle}
\def\<{\left\langle}
\def\>{\right\rangle}
\def\div{\mathrm{div}\hspace{0.5mm}}
\def\divh{\mathrm{div}_h\hspace{0.5mm}}
\def\curl{\mathrm{curl}\hspace{0.5mm}}
\def\curlh{\mathrm{curl}_h\hspace{0.5mm}}
\def\bcurl{\mathbf{curl}\hspace{0.5mm}}

\def\CR{\mathcal{CR}}

\def\R{\mathbb R}
\def\O{\Omega}
\def\Ob{\overline{\Omega}}
\def\Oh{\widehat\Omega}
\def\Ohh{\Oh_h}
\def\Qt{\widetilde{Q}}
\def\p{\partial}
\def\Th{\mathcal{T}_h}
\def\TM{\mathcal{T}_M}
\def\T2h{\mathcal{T}_{2h}}
\def\TTh{\mathscr{T}_h}
\def\N{\mathcal{N}}
\def\Rh{\mathcal{R}_h}
\def\Kh{\mathcal{K}_h}
\def\Ro{\mathcal{R}^{\circ}}
\def\Ko{\mathcal{K}^{\circ}}
\def\Ih{\mathcal{I}_h}

\def\Qhz{\mathcal{Q}_{1,h}(\O)}
\def\Xh{[\NChz]^2}
\def\Mhz{\mathcal{P}_{0,h}(\O)}
\def\Mh{\mathcal{P}'_{0,h}(\O)}
\def\hph{\hat{p_h}}
\def\NC{\mathcal{NC}}
\def\NCh{{\mathcal{NC}}^h(\O)}
\def\NChz{{\mathcal{NC}}^h_0(\O)}
\def\Chz{{\mathcal C}^h_0(\O)}
\def\B{\mathcal{B}}
\def\xt{\hat{x}}
\def\yt{\hat{y}}
\def\lam{\lambda}
\def\bQ{\mathbf{b}_Q}
\def\nQ{\mathbf{n}_Q}
\def\Qrt{Q^{rt}}
\def\Qlt{Q^{lt}}
\def\Qlb{Q^{lb}}
\def\Qrb{Q^{rb}}
\def\Vrt{\V^{rt}}
\def\Vlt{\V^{lt}}
\def\Vlb{\V^{lb}}
\def\Vrb{\V^{rb}}
\def\b{\mathbf{b}}
\def\n{\mathbf{n}}
\def\sq{\square}

\def\u{\mathbf{u}}
\def\v{\mathbf{v}}
\def\bu{\mathbf{u}}
\def\bv{\mathbf{v}}
\def\bw{\mathbf{w}}
\def\V{\mathbf{V}}
\def\C{\mathbf{C}}
\def\vh{\mathbf{v}_h}
\def\wh{\mathbf{w}_h}
\def\uh{\mathbf{u}_h}
\def\zh{\mathbf{z}_h}
\def\Psih{\mathbf{\Psi}_h}
\def\pih{\pi_h}
\def\Pih{\Pi_h}
\def\Honez{H_0^1(\Omega)}
\def\alp{\alpha}
\def\bet{\beta}
\def\gam{\gamma}
\def\del{\delta}
\def\ori{\mathbf{0}}
\def\sh{\sqrt{h}}
\def\x{\mathbf{x}}
\def\y{\mathbf{y}}
\def\z{\mathbf{z}}
\def\pron{\Pi_h}
\def\eps{\epsilon}

\def\NChzvec{[\NChz]^2}
\def\Wtinfvec{[W^{2,\infty}(\O)]^2}
\def\vtwonorm{ \snorm{\v}{[W^{2,\infty}(\O)]^2}}
\def\curlvnorm{ \Lnorm{\curl\v}{\O}}
\def\c{\mathbf{c}}
\def\rhomin{\rho_{\min}}
\def\cc{\mathscr{P}}
\def\S{\mathscr{F}}
\def\f{\mathbf{f}}
\def\Lip{\mbox{\textnormal{\tiny{Lip}}}}

\def\Pht{\widehat{P_h}(\Omega)}
\def\pht{\hat{p_h}}
\def\qht{\hat{q_h}}
\def\qhx{q_h^{\mbox{x}}}
\def\Phx{P_h^{\mbox{x}}(\Omega)}
\def\phb{p_h^{\mbox{x}}}
\def\Qh{\widetilde{Q}}
\def\VV{\mathcal{V}}
\def\HH{\mathcal{H}}
\def\XX{\mathcal{X}}
\def\bh{\mathbf{b}_h}
\def\phl{p_h}
\def\Wpho{\widetilde{P_h}(\O)}

\newcommand{\bff}{\mathbf{f}}

\newcommand {\snorm}[2] {| #1 |_{#2}}
\newcommand {\Hnorm}[2] {| #1 |_{#2}}
\newcommand {\norm}[2] {\| #1 \|_{#2}}

\def\nx{\n_{\x}}
\def\k{\mathbf{k}}
\def\ble{\ \le \ }
\def\bge{\ \ge \ }
\def\beq{\ = \ }
\def\ssubset{\ \subset\ }
\def\CO{C_{\O}}
\newcommand{\gCQ}[1]{\big(\C(Q^{#1})\big)}
\newcommand {\squad}{\hspace{2mm}}

\def\mmskip{\vspace{1mm}}
\def\tele{\mathbf{w}}

\maketitle
\begin{abstract}
Recently, 
the $P_1$-nonconforming finite element space over square meshes has been proved 
stable to solve Stokes equations with the piecewise constant space 
for velocity and pressure, respectively.
In this paper, we will introduce its locally divergence-free subspace 
to solve the elliptic problem for the velocity only
decoupled from the Stokes equation.
The concerning system of linear equations is much smaller 
compared to the Stokes equations.
Furthermore, it is split into two smaller ones.
After solving the velocity first,
the pressure in the Stokes problem can be obtained by an explicit method very rapidly.
\end{abstract}

\section{Introduction}
A divergence-free vector field frequently appears in various  
mathematical and engineering problems
such as an incompressible flow in the Navier-Stokes equation or
a solenoidal magnetic induction in the Maxwell equations or
the limit of displacements in elasticity equations when Poisson's ratio goes to 1/2.

An incompressible Stokes problem can be reduced to
an elliptic problem for the velocity only in the divergence-free space 
\cite{Brezzi1991, GR, Pironneau}.
The locally divergence-free subspace of $[\CR]^2$
 was used for finite element methods
to solve that elliptic problem  \cite{Brezzi1991, thomasset},
where $\CR$ is the Crouzeix-Raviart $P_1$-nonconforming 
finite element space on triangular meshes.
It have also been adopted for the time-harmonic Maxwell equations \cite{Brenner2007}.

It contains enough interpolants 
to approximate continuous divergence-free functions in $[H^2]^2$,
since 
it can be interpreted as the curl of the
Morley element.
If the domain is simply connected in $\R^2$,
its dimension is the number of interior vertices and edges \cite{Hecht, thomasset},
which is about two third of that of $[\CR]^2$.

A conforming locally divergence-free space whose elements are 
piecewise linear can be constructed with the curls of $C^1$-Powell-Sabin elements
 on triangular meshes for biharmonic problems \cite{Powell}.
We can find how to construct the  locally divergence-free subspace 
for various finite element spaces \cite{Ye1992, Ye1997}.
Instead of working with divergence-free spaces, some researchers have developed 
 finite element methods for Stokes equations
 whose velocity solutions are resulted divergence-free 
\cite{Zhang2011, Zhang2009}
as well as locally divergence-free discontinuous Galerkin methods
\cite{Cockburn2004, Cockburn2007}, multigrid methods \cite{Austin}
and isogeometric analysis \cite{Buffa, Evans}.

In this paper, we are interested in the locally divergence-free subspace 
of $[\NC]^2$, the $P_1$-nonconforming finite element space on square meshes. 
The space $\NC$ consists of functions which are linear in each square 
and continuous on each midpoint of edge \cite{parkthesis, p1quad}.
Recently, it has been proved that  $[\NC]^2$
is stable to solve Stokes equations with the piecewise constant space 
for velocity and pressure, respectively \cite{kim2016}.

We will apply the locally divergence-free subspace
to solve the elliptic problem for the velocity only,
reduced from the incompressible Stokes problem.
The concerning system of linear equations is much smaller than
that of the Stokes equation. 
Furthermore, if we divide the squares in the mesh into the red and black squares 
like a checkerboard, the curl of divergence-free element has
its support in red squares only, otherwise black ones only.
Thus, the system from the elliptic problem is split into two
independent smaller ones.
After solving the velocity first,
the pressure in the Stokes problem 
will be obtained by an explicit method very rapidly.

The paper is organized as follows.
In the next section the $P_1$-nonconforming finite element space
on quadrilateral meshes will be briefly reviewed.
Then, restricted on  square meshes, 
we will devote section \ref{sec:divfree} to characterizing
its locally divergence-free subspace as well as a basis. 
In section \ref{sec:Stokes},
 the reduced elliptic problem for the velocity
and an explicit method for the pressure in the Stokes problem 
are stated, respectively.
Finally, some numerical tests will be presented in the last section.

Throughout the paper, 
$C_S$ is a generic notation for a positive constant which depends only on $S$.

\section{ $P_1$-nonconforming quadrilateral finite element}
Let $\O$ be a simply connected polygonal domain in $\R^2$ with
a triangulation $\Th$ which consists of uniform squares of width and
height $h$ as in Figure \ref{fig:Th}. 
For a vertex or edge in $\Th$, we call them a boundary vertex or edge 
if they belong to $\p\O$, otherwise, an interior vertex or edge. 

The $P_1$-nonconforming quadrilateral finite element
spaces \cite{parkthesis, p1quad} are defined by  
\vspace{1mm}
\begin{equation*}
\begin{array}{c}
\vspace{1mm}
\NCh = \{v_h\in L^2(\O)\ :\  v_h|_Q \in\, <1,x,y >
 \mbox{ for all squares } Q\in\Th,\,
\qquad\qquad\qquad\qquad \\ 
\vspace{3mm}
\hspace{4cm}
 v_h \mbox{ is continuous at every midpoints of edges in }
 \Th \},\\
\NChz= \{v_h\in \NCh \ : \ v_h(m) = 0 \text{ for all midpoints } m 
\mbox{ of boundary edges in } \Th \}.
\end{array}
\end{equation*} 

\begin{figure}[ht]
\hspace{4.8cm}
\includegraphics[width=0.4\textwidth]{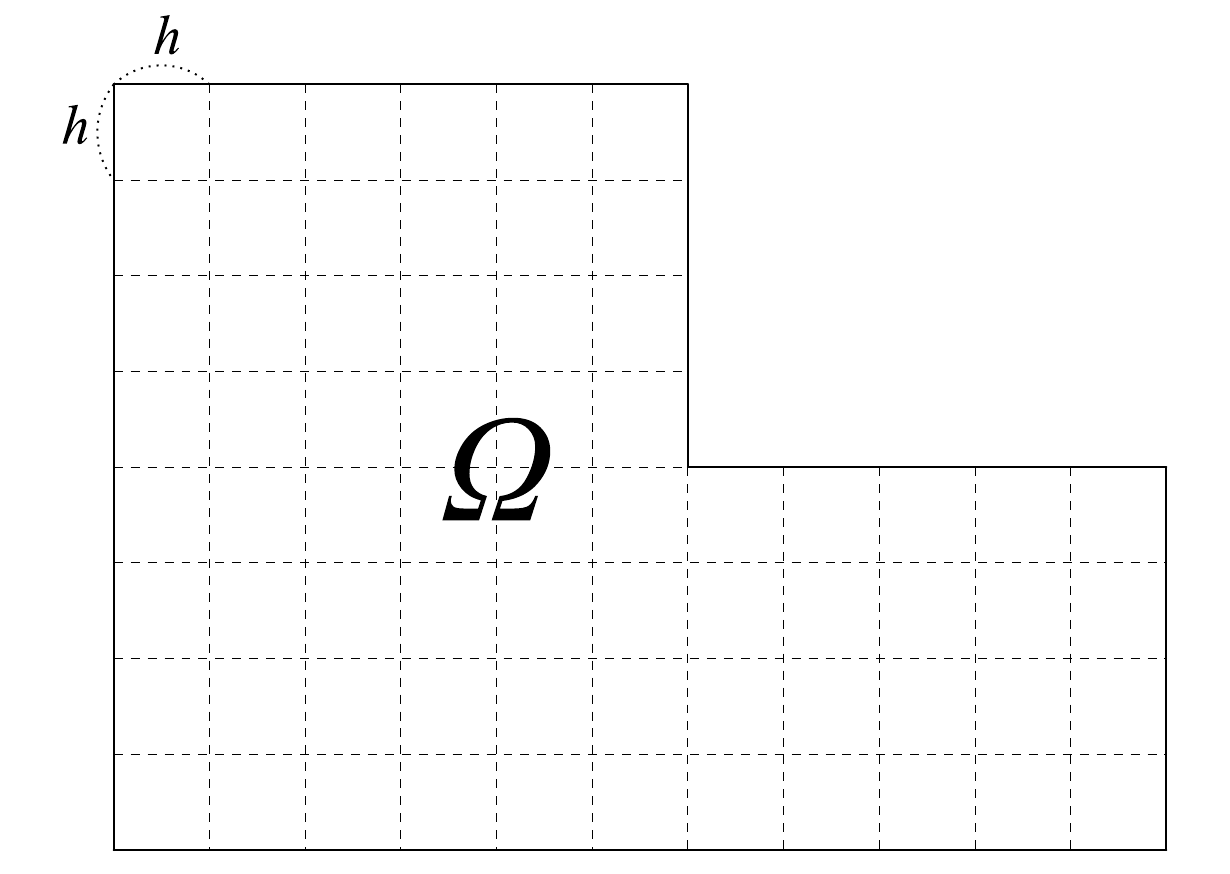}
\caption{Triangulation $\Th$ of $\Ob$ into uniform squares
}\label{fig:Th}
\end{figure} 
\vspace{2mm}

\noindent
When we assign a value $c_m$ for the midpoint $m$ of each edge in $\Th$, 
 there exists $v_h\in\NCh$ such that
$v_h(m)=c_m$ at all midpoints $m$ of edges in $\Th$ if and only if
\begin{equation}\label{eq:condnch}
c_{m_1}+c_{m_3} = c_{m_2}+c_{m_4},
\end{equation}
whenever $ m_1,\cdots,m_4$ are of clockwisely numbered edges of 
a square $Q\in \Th.$

\mmskip
For each vertex $\V$ in $\Th$, define a function $\psi^{\V}\in \NCh$ by
its values at all midpoints $m$ of edges in $\Th$ such that
\begin{equation}\label{def:psiV}
\psi^{\V}(m)=\left\{\begin{array}{cl}
1, & \mbox{ if } m \mbox{ belongs to an edge which meets } \V, \\
0, & \mbox { otherwise. }
\end{array}\right.
\end{equation}
Note $\psi^\V \in\NCh$ is well defined since its values at the midpoints satisfy
the condition \eqref{eq:condnch}.
Then, we have a basis for $\NChz$ as 
\begin{equation}\label{eq:spannchz}
 \NChz=\mbox{Span}\{ \psi^{\V} \in \NChz \ : \ \V 
\mbox{ is an interior vertex in } \Th \}. 
\end{equation}
Hence, the dimension of $\NChz$ is the number of interior vertices in $\Th$
\cite[Theorem 2.5]{p1quad}.
  
Define an interpolation $\pih: H^1_0(\O)\cap C(\O)\rightarrow \NChz$ as
\begin{equation}\label{def:pih}
 \pih v = \sum_{\V} \frac{v(\V)}2 \psi^{\V},\quad \forall v\in H^1_0(\O)\cap C(\O),   
\end{equation}
where the summation runs over all interior vertices $\V$ in $\Th$. 
We note $\pih v$ satisfies
 \begin{equation}\label{eq:pih}
 \pih v \Big(  \frac{\V_1+\V_2}2 \Big) \beq \frac{ v(\V_1)+v(\V_2)}2 \quad 
\mbox{ for all adjacent vertices } \V_1, \V_2 \mbox{ in } \Th.
\end{equation}
Then, the interpolation error is estimated by
\begin{equation*}\label{eq:ncintpolerr}
  \norm{v -  \pih v}{L^2(\O)} + h \snorm{v -  \pih v}{1,h}
\ble \CO\, h^2 \snorm{v}{H^2(\O)}, \quad  
\forall v\in H^1_0(\O)\cap H^2(\O),
\end{equation*}
where 
$\snorm{\Box}{1,h}$ denotes 
the mesh-dependent discrete $H^1$-seminorm such that
\begin{equation}\label{def:dnorm}
 |\Box|_{1,h} = \left[\sum_{Q\in\Th} \int_Q |\nabla \Box|^2\,dxdy \right]^{1/2}. 
\end{equation}

Let  $\Qhz$ be the conforming piecewise bilinear space  over $\Th$ as 
\begin{equation*}
 \Qhz =\{ w_h\in H_0^1(\O)\cap C(\O)\ :\  
w_h|_Q \in\, <1,x,y,xy> \mbox{ for all squares } Q \in\Th\},
\end{equation*}
with an assumption that $\p\O$ consists of segments parallel to
$x,y$-axes. Two lowest order finite element spaces 
$\Qhz$ and $\NChz$ are isomorphic as in the following lemma.
\begin{lemma}\label{lem:isoQ1P1}
Let $\pih: \Qhz \rightarrow \NChz$
be the  operator as in \eqref{def:pih}.
Then $\pih$ is a bijection.
\end{lemma}
\begin{proof}
Let $\pih v\in\NChz$ vanish at all midpoints of edges in $\Th$ for some $v\in\Qhz$.
Then,
$v$ vanishes on each interior edges $E$ in $\Th$ 
which meets $\p \O$, since $v$ is linear on $E$ and 
 vanishes at the two endpoints of $E$  by \eqref{eq:pih}.

By a sweeping out argument, we conclude that $v$
vanishes on all edges in $\Th$. This means
that $\pih$ is a bijection
since the dimensions of $\Qhz$ and $\NChz$ are same 
as the number of interior vertices in $\Th$.
\end{proof}

For a vector valued function $\v=(v_1, v_2)$, the component-wise interpolation and 
discrete $H^1$-seminorm   from \eqref{def:pih} and \eqref{def:dnorm} will be used  as  
\begin{equation*}\label{eq:vecvers}
 \Pih\v = (\pih v_1, \pih v_2), \quad \snorm{\v}{1,h} 
=\Big( \snorm{v_{1}}{1,h}^2 
+\snorm{v_{2}}{1,h}^2 \Big)^{1/2}.
\end{equation*}
Let $\divh, \curlh$ be the mesh-dependent discrete divergence and curl as
\begin{equation*}
(\divh \v)\big|_Q  =  \div (\v|_Q) , \quad (\curlh \v)\big|_Q =  \curl (\v|_Q), 
\quad \forall Q\in\Th.
\end{equation*}

It is well known there is a constant $\CO$ such that \cite{GR}
\begin{equation}\label{ieq:convers}
\snorm{\v}{[H^1(\O)]^2} 
\le \CO ( \norm{\div\v}{L^2(\O)} +  \norm{\curl\v}{L^2(\O)} ),
\quad \forall\v\in [H^1_0(\O)]^2.
\end{equation}
We have a similar result to \eqref{ieq:convers} for $\NChzvec$
 in the following lemma.  
\begin{lemma} \label{lem:normdecom}
For all $\v_h\in\NChzvec$, we have
\begin{equation*}
\snorm{\v_h}{1,h}^2 = 
 \norm{\divh \v_h}{L^2(\O)}^2 + \norm{\curlh \v_h}{L^2(\O)}^2.
\end{equation*}
\end{lemma}
\begin{proof}

For each $\wh=(w_{1}, w_{2}) \in [\Qhz]^2$, we have, by integration by parts,
\begin{equation*}
\begin{array}{ccl}
\disp\int_\O w_{1x}  w_{2y}\ dxdy &=& \disp\sum_{Q\in\Th}  \int_Q w_{1x}  w_{2y}\ dxdy   \\
&=&  \disp\sum_{Q\in\Th} \int_{\p Q} w_{1x} w_2 \nu_y \ ds -   \int_Q w_{1xy}  w_{2}\ dxdy 
= - \disp\sum_{Q\in\Th}     \int_Q w_{1xy}  w_{2}\ dxdy,
\end{array} 
\end{equation*}
since $\nu_y$ is nonzero only at the horizontal edges of $\p Q$ and $w_{1x}, w_2$ 
are continuous there.
The same argument is repeated to get
\begin{equation}\label{eq:w12xy}
\begin{array}{ccl}
 \disp\int_\O w_{1y}  w_{2x}\ dxdy 
&=&\disp \sum_{Q\in\Th} \int_{\p Q} w_{1y} w_2 \nu_x \ ds -  
 \int_Q w_{1yx}  w_{2}\ dxdy\\
&=& - \disp\sum_{Q\in\Th}  
   \int_Q w_{1yx}  w_{2}\ dxdy = \int_\O w_{1x}  w_{2y}\ dxdy.
\end{array}
\end{equation}
Then, from \eqref{eq:w12xy}, we can establish
\begin{equation}\label{eq:whhold}
\begin{array}{lll}
\vspace{2mm}
\snorm{\wh}{[H^1(\O)]^2}^2
&=&(w_{1x},w_{1x})+(w_{1y},w_{1y})+(w_{2x},w_{2x})+(w_{2y},w_{2y}) \\
\vspace{2mm}
&=&(w_{1x}+w_{2y}, w_{1x}+w_{2y}) +(w_{2x}-w_{1y},w_{2x}-w_{1y}) \\
&=& \norm{\div \wh}{L^2(\O)}^2 + \norm{\curl \wh}{L^2(\O)}^2.
\end{array}
\end{equation}

Now, if $\vh=(v_1, v_2)\in [\NChz]^2$,
there exists $\wh=(w_1, w_2)\in [\Qhz]^2$  by Lemma \ref{lem:isoQ1P1}  such that
\begin{equation*}\label{eq:dvi}  v_i = \pih w_i, \quad i=1,2. \end{equation*}
Given square $Q\in\Th$, assume that
\begin{equation}\label{eq:dwh}
 w_i(x,y)= \alp_i+ \bet_i \xt + \gam_i \yt + \del_i \xt\yt, 
\quad  i=1,2, 
\end{equation}
where $\xt=x-c_1,\ \yt=y-c_2$ for the center point $(c_1,c_2)$ of $Q$. Note 
\begin{equation}\label{eq:xhyh}
 \int_Q \xt \ dxdy = \int_Q \yt\ dxdy = \int_Q \xt\yt \ dxdy =0. 
\end{equation}

By \eqref{eq:pih}, $\pih \xt\yt$ vanishes at all midpoints of edges in $\Th$ 
since the values of $\xt\yt$ differ only by their signs
at every two endpoints of an edge of $Q$.
It means
 $v_i$ is the linear part of $w_i$ so that
\begin{equation}\label{eq:dvh}  
v_i(x,y)= \pih w_i(x,y)=\alp_i+ \bet_i \xt + \gam_i \yt , 
\quad \ i=1,2.
\end{equation}

From \eqref{eq:dwh}, \eqref{eq:xhyh}, \eqref{eq:dvh}, we expand that
\begin{equation*}
\begin{array}{l}
\mmskip
\snorm{\wh}{[H^1(Q)]^2}^2 - \norm{\div \wh}{L^2(Q)}^2 - \norm{\curl \wh}{L^2(Q)}^2 \\
\qquad\qquad = \disp\int_Q (\bet_1+\del_1\yt)^2+ (\gam_1+\del_1\xt)^2 +(\bet_2+\del_2\yt)^2
+ (\gam_2+\del_2\xt)^2 \ dxdy \\
\qquad\qquad\qquad
 \disp -\int_Q (\bet_1+\del_1\yt + \gam_2+\del_2\xt)^2 + (\bet_2+\del_2\yt -\gam_1 -\del_1\xt)^2 
\ dxdy \\
\mmskip
\qquad\qquad=\disp\int_Q \bet_1^2+ \gam_1^2 +\bet_2^2+ \gam_2^2 - (\bet_1 + \gam_2 )^2 - 
(\bet_2 -\gam_1)^2 \ dxdy \\  
\qquad\qquad=\snorm{\vh}{[H^1(Q)]^2}^2 - \norm{\div \vh}{L^2(Q)}^2 - \norm{\curl \vh}{L^2(Q)}^2.
\end{array}
\end{equation*}
It completes the proof, since $\wh$ satisfies \eqref{eq:whhold}. 
\end{proof}

\section{$P_1$- nonconforming divergence-free space}\label{sec:divfree}
Let $V_h$ be a locally divergence-free subspace of $\Xh$ as 
\begin{equation}\label{eq:divh0}
V_h=\{\v_h\in [\NChz]^2\ :\ \divh \v_h =0  \}. 
\end{equation}

Throughout the remaining of the paper, 
in order to exclude a pathological triangulation $\Th$ such as a ladder
or a union of two rectangles whose intersection is merely one square or edge in $\Th$,
we assume that
\begin{assumption}\label{assumth}
No square has 4 boundary vertices. No interior edge meets 2 boundary vertices. 
If a square has only two boundary vertices, 
they are the two endpoints of one edge.
\end{assumption}

Let $\mathcal{P}_{0,h}(\O)$ be a space of piecewise constant functions as
\[\mathcal{P}_{0,h}(\O) =\{q_h \in L^2(\O)\ :\ q_h|_Q = <1> \mbox{ for all } Q\in\Th\}. \]
The value of a function $q_h\in \Mhz$
at a square $Q\in\Th$ will be abbreviated to $q_h(Q)$.

For an interior vertex $\V$, let 
$Q_1, Q_2, Q_3, Q_4$ be the squares in $\Th$ which meet $\V$, 
counterclockwisely numbered from the square whose left bottom vertex is
$\V$ as in Figure \ref{fig:psivab}-(a).
Using the scalar basis function $\psi^\V\in \NChz$ in \eqref{def:psiV}, 
for a vector value $(a,b)$,
define a function $\psi^{\V}[a,b] \in [\NChz]^2$ by its values at
all midpoints $m$ of edges in $\Th$ such that
\begin{equation*}
\psi^{\V}[a,b](m)=\left\{\begin{array}{cl}
(a,b), & \mbox{ if } m \mbox{ belongs to an edge which meets } \V, \\
(0,0), & \mbox { otherwise. }
\end{array}\right.
\end{equation*}
We can easily check that 
$\divh \psi^{\V}[a,b],\ \curlh \psi^{\V}[a,b] \in \Mhz$
have nontrivial values at only 4 squares in $\Th$, 
given as in Figure \ref{fig:psivab}-(b), (c),
\begin{equation}\label{eq:divvab}
\divh\psi^{\V}[a,b](Q_j) =
\left\{\begin{array}{rl}
-(a+b)/h  \quad &\mbox{ if } j=1, \\   
(a-b)/h  \quad &\mbox{ if } j=2, \\   
(a+b)/h  \quad &\mbox{ if } j=3, \\   
(b-a)/h  \quad &\mbox{ if } j=4, 
\end{array}
\right.
\end{equation}
\begin{equation}\label{eq:curlvab}
\curlh\psi^{\V}[a,b](Q_j) =
\left\{\begin{array}{rl}
(a-b)/h  \quad &\mbox{ if } j=1, \\   
(a+b)/h  \quad &\mbox{ if } j=2, \\   
(b-a)/h  \quad &\mbox{ if } j=3, \\   
-(a+b)/h  \quad &\mbox{ if }j=4. 
\end{array}
\right.
\end{equation}
\def\psivabxxx{\psi^{\V}[a,b]}
\begin{figure}[hb]
\hspace{20mm}
\subfigure[ $\psivabxxx$  ]{
\includegraphics[width=0.18\textwidth]{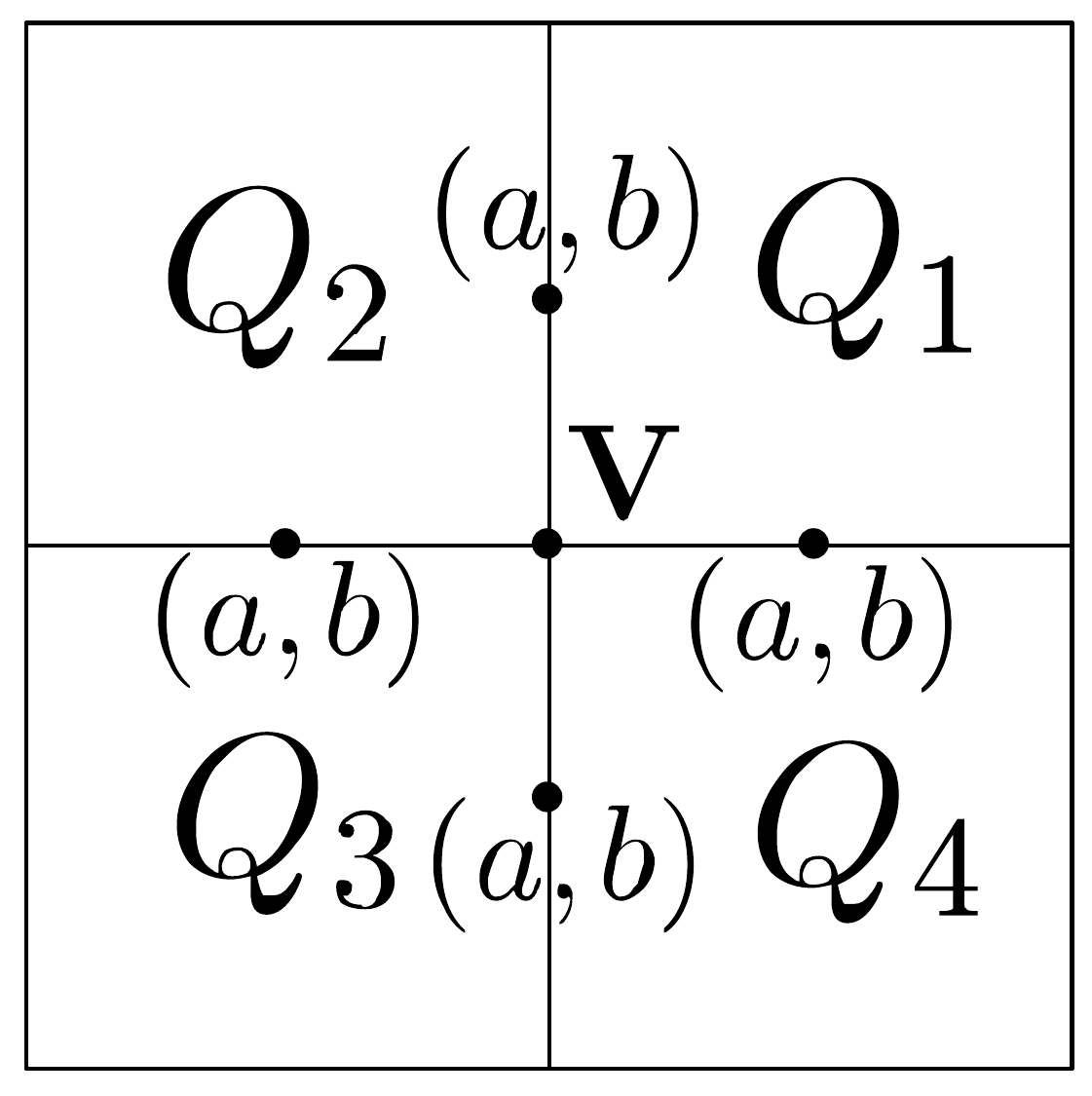}
}
\qquad
\subfigure[$\divh \psivabxxx$ ]{
\includegraphics[width=0.18\textwidth]{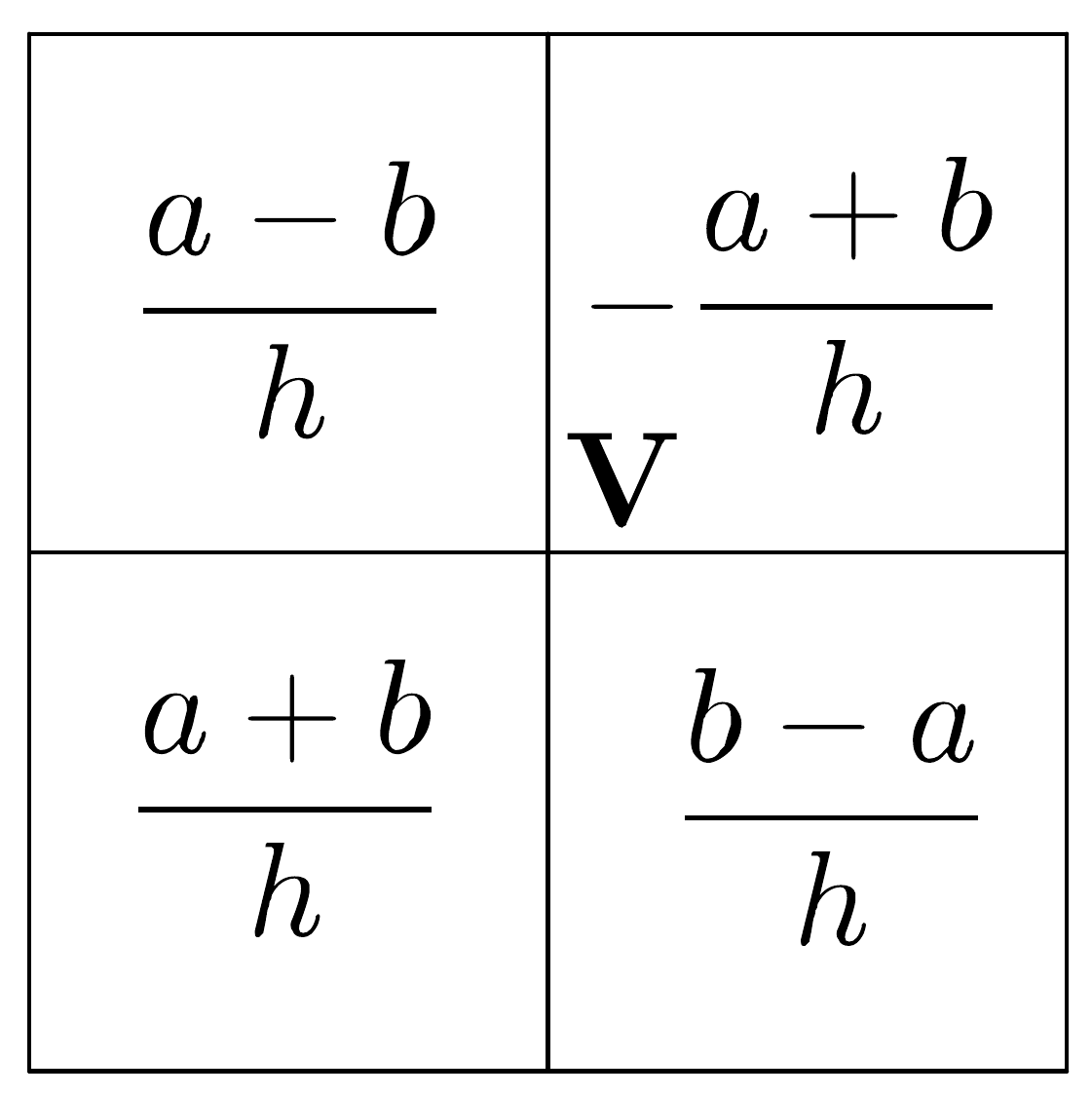}
}
\qquad
\subfigure[$\curlh\psivabxxx $ ]{
\includegraphics[width=0.18\textwidth]{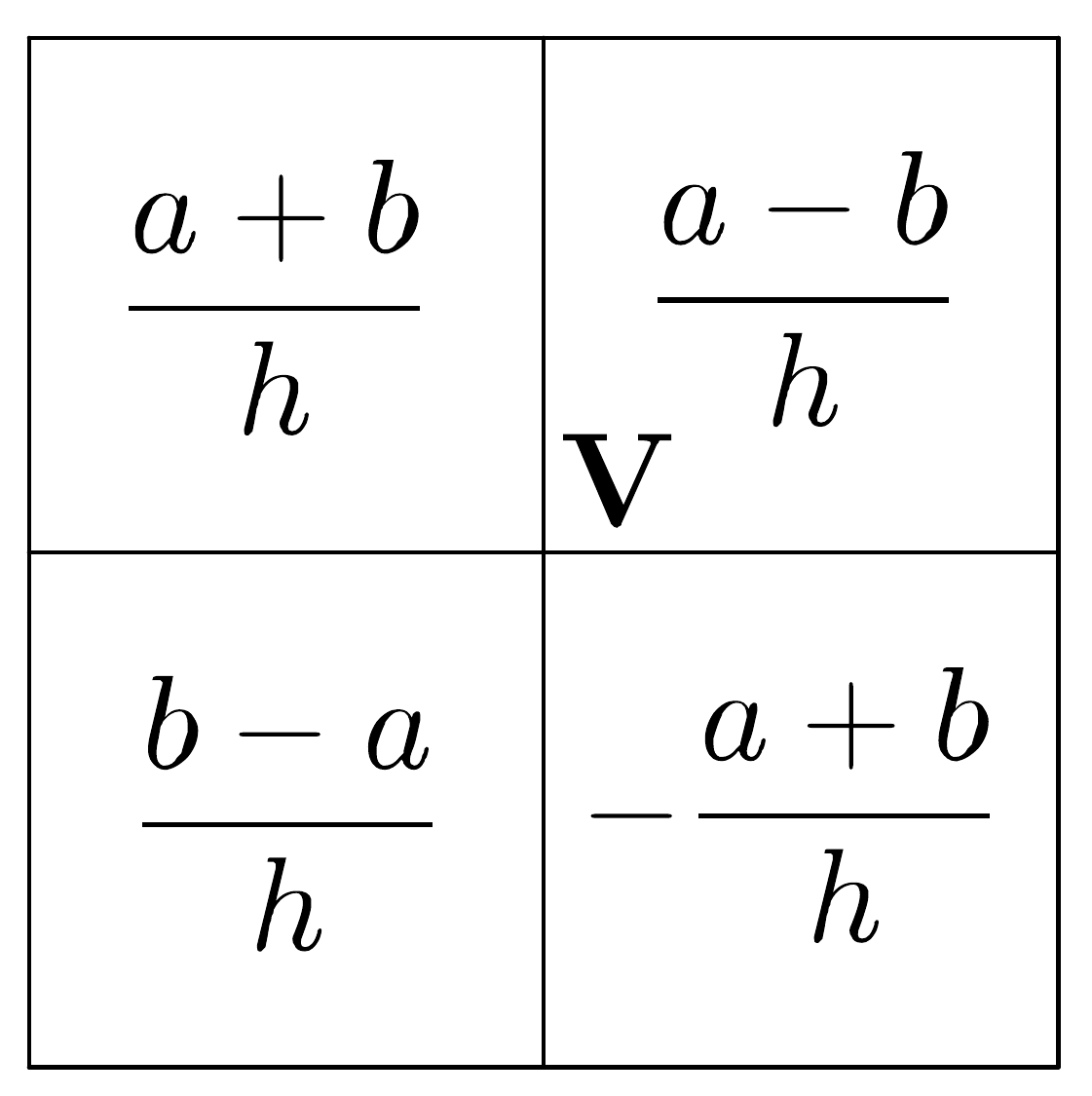}
}
\caption{ $\psi^{\V}[a,b]\in\NChzvec$ and its divergence and curl 
}\label{fig:psivab}
\end{figure} 

\subsection{Dimension of $V_h$}
We call a square $Q\in\Th$ an interior square 
if it has 4 interior vertices, otherwise, a boundary square. 
Let $\Ih$ be a set of all interior squares in $\Th$ and
denote by  $\#\Ih$ the number of its elements.
In the following lemma,  $\#\Ih$ relates with other numbers depending on $\Th$.
 
\begin{lemma}\label{lem:euler}
Let $\N(\V^i),\ \N(Q) $ be the numbers of all interior vertices and all squares in $\Th$, 
respectively.  Then,
\begin{equation}\label{eq:numIh}
 \#\Ih \beq 2\ \N(\V^i) - \N(Q) + 2.
\end{equation}
\end{lemma}
\begin{proof}
Let $\N(E^i)$  be the number of all interior edges in $\Th$.
Note the number of all boundary vertices, denoted by $\N(\V^b)$, is same as that of
all boundary edges,  denoted by $\N(E^b)$.

Every squares in $\Th$ has  4 edges. When we count them all 
to make $4\N(Q)$, 
each interior edges is done twice whereas each boundary edge is done once.
That is,
\begin{equation}\label{eq:Ihc1}
 4\N(Q) = 2\N(E^i)+\N(E^b).
\end{equation}

Let $k_j$ be the number of all vertices that meet only $j$ squares in $\Th$ 
for $j=1,2,3,4$. 
Then, $k_4=\N(\V^i)$ and $k_2$ is the number of boundary vertices that are not corners,
while
$k_1,k_3$ designate those of corners of respective inner angles 
$90^\circ, 270^\circ$. 

While we count every 4 vertices of a square in $\Th$ to make $4\N(Q)$, 
each vertex $\V$ is done repeatedly by its number of squares which meet $\V$.
Thus we have,
\begin{equation}\label{eq:Ihc2}
 4\N(Q)= k_1 + 2k_2+3k_3 + 4k_4. 
\end{equation}
From \eqref{eq:Ihc1} subtracted by \eqref{eq:Ihc2}, 
\begin{equation}\label{eq:Ihc3}
 2\N(E^i)-\N(\V^b)-4\N(\V^i)+k_1-k_3 =0,
\end{equation}
since $k_1+k_2+k_3 =\N(\V^b)=\N(E^b)$.

If we paint one boundary square for 
every one boundary edge, each boundary square that meets a corner of 
angle $90^\circ$ is done twice by Assumption \ref{assumth}, while 
each corner of angle  $270^\circ$ remains one boundary square unpainted. 
It means that 
\begin{equation}\label{eq:Ihc4}
 \N(Q)-\#\Ih=\N(E^b)-k_1+k_3. 
\end{equation}
From \eqref{eq:Ihc4} added by  \eqref{eq:Ihc3}, we obtain \eqref{eq:numIh}
through the following Euler formula for simply connected domain:
\begin{equation*}\label{eq:Euler}
\N(Q) - \N(E^i) + \N(\V^i)  = 1.
\end{equation*}
\end{proof}

The two-color theorem guarantees that the squares in $\Th$ can be colored
in two colors, if each interior vertex meets with even number of edges 
\cite[Ch. 15]{Stein}.
Thus, the squares in $\Th$ are grouped 
into $\Rh$ of the red squares and $\Kh$ of the black ones so that
squares sharing at least one edge have different colors,
as a checkerboard in Figure \ref{fig:redblack}.
\begin{figure}[ht]
\hspace{4.6cm}
\includegraphics[width=0.4\textwidth]{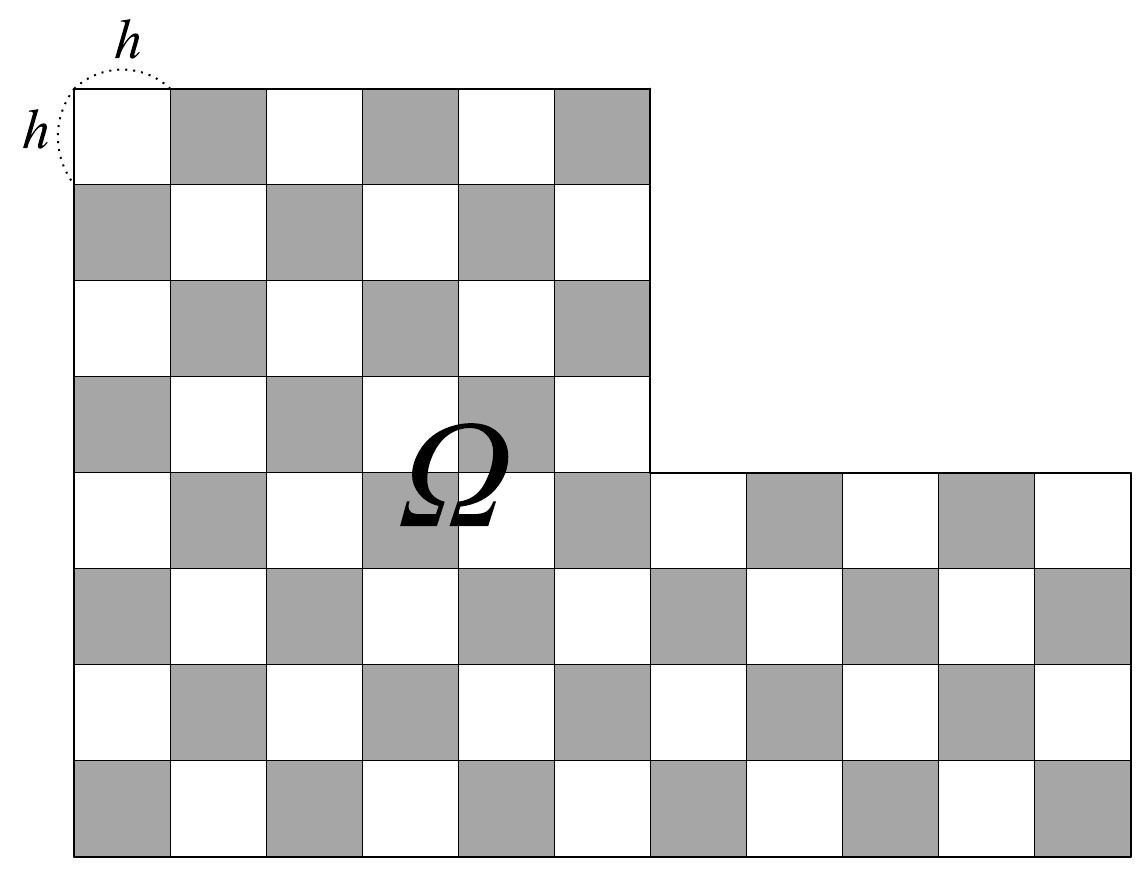}
\caption{ Red and black checkerboard in $\Th$
}\label{fig:redblack}
\end{figure} 

In the following lemma, if two squares have same color, they are connected by a path which
passes only squares of that color and does not meet any boundary vertices.
Let 
\[ \Ro=\Big(\disp\bigcup_{Q\in\Rh} \overline{Q}\Big)\setminus \p\O,\quad
   \Ko=\Big(\disp\bigcup_{Q\in\Kh} \overline{Q}\Big)\setminus \p\O.   \]
\begin{lemma}\label{lem:connecRB}
$\Ro$ is path-connected and so is $\Ko$.
\end{lemma}
\begin{proof}
It is enough to prove for $\Ro$. 
If $\x,\y \in\Ro$,  there is a path $\cc$ in $\O$ joining $\x,\y$,
since the open set $\O$ is connected.
We can repair $\cc$ into a path in $\Ro$ in the following way. 

Let $\c_1, \c_2$ be two points in the path $\cc$ which are
on the boundary of a black square $Q_K$  and 
the part of $\cc$ between them, named $\cc_{12}$, belong to the interior of $Q_K$. 
Denote by $E_1, E_2$, the edges of $Q_K$ on which $\c_1, \c_2$ are, respectively.
The following 3 cases  are possible for $E_1, E_2$.
 
\vspace{1mm}
Case I. 
If $E_1 \beq E_2$, 
we can easily repair $\cc_{12}$ into the segment in $\Ro$ between $\c_1, \c_2$.

\vspace{1mm}
Case II. 
Let $E_1, E_2$ meet at a vertex $\V$ of $Q_K$. 
If $\V$ is an interior vertex,  
$\cc_{12}$ can be repaired into the 2 segments via $\V$ in $\Ro$ as in Figure \ref{fig:chgpath}-(a).
When $\V$ is a boundary vertex, by Assumption \ref{assumth}, 
the other three vertices of $Q_K$ 
are all interior vertices, since $\c_1,\c_2$ belong to $\O$.
Thus, 
$\cc_{12}$ can be done into the 4 segments in $\Ro$ which do not pass $\V$ 
as in Figure \ref{fig:chgpath}-(b).

\vspace{1mm}
Case III. If $E_1, E_2$ are parallel,
by Assumption \ref{assumth},
there are endpoints $\V_1, \V_2$ of $E_1, E_2$, respectively,
which are interior vertices 
such that the segment between them is an edge of $Q_K$. Thus,
$\cc_{12}$ can be done into  
the 3 segments via $\V_1, \V_2$ in $\Ro$ as in Figure \ref{fig:chgpath}-(c).

\begin{figure}[hb]
\hspace{10mm}
\subfigure[$\overline{E_1}\cap\overline{E_2}$ is a  vertex in $\O$   ]{
\includegraphics[width=0.25\textwidth]{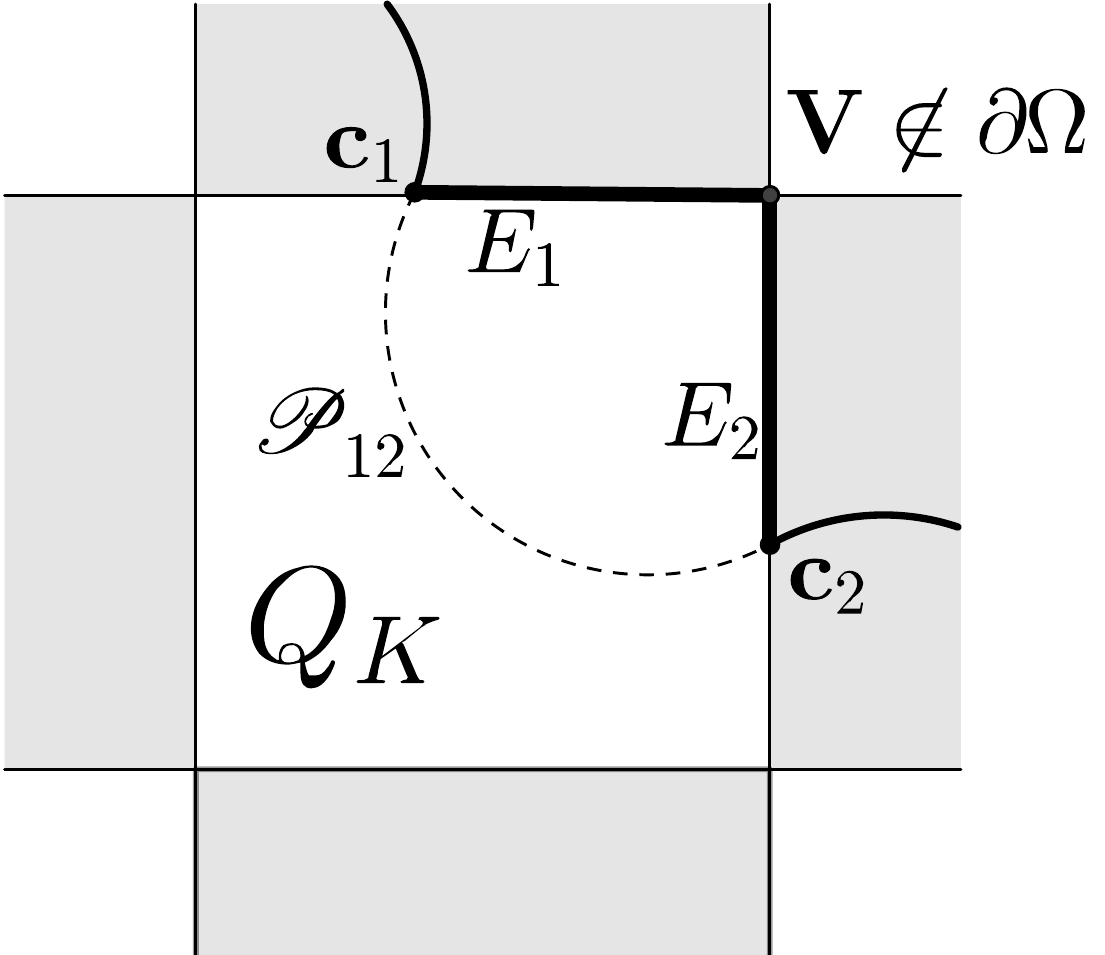}
}
\qquad
\subfigure[$\overline{E_1}\cap\overline{E_2}$ is a vertex on $\p\O$   ]{
\includegraphics[width=0.25\textwidth]{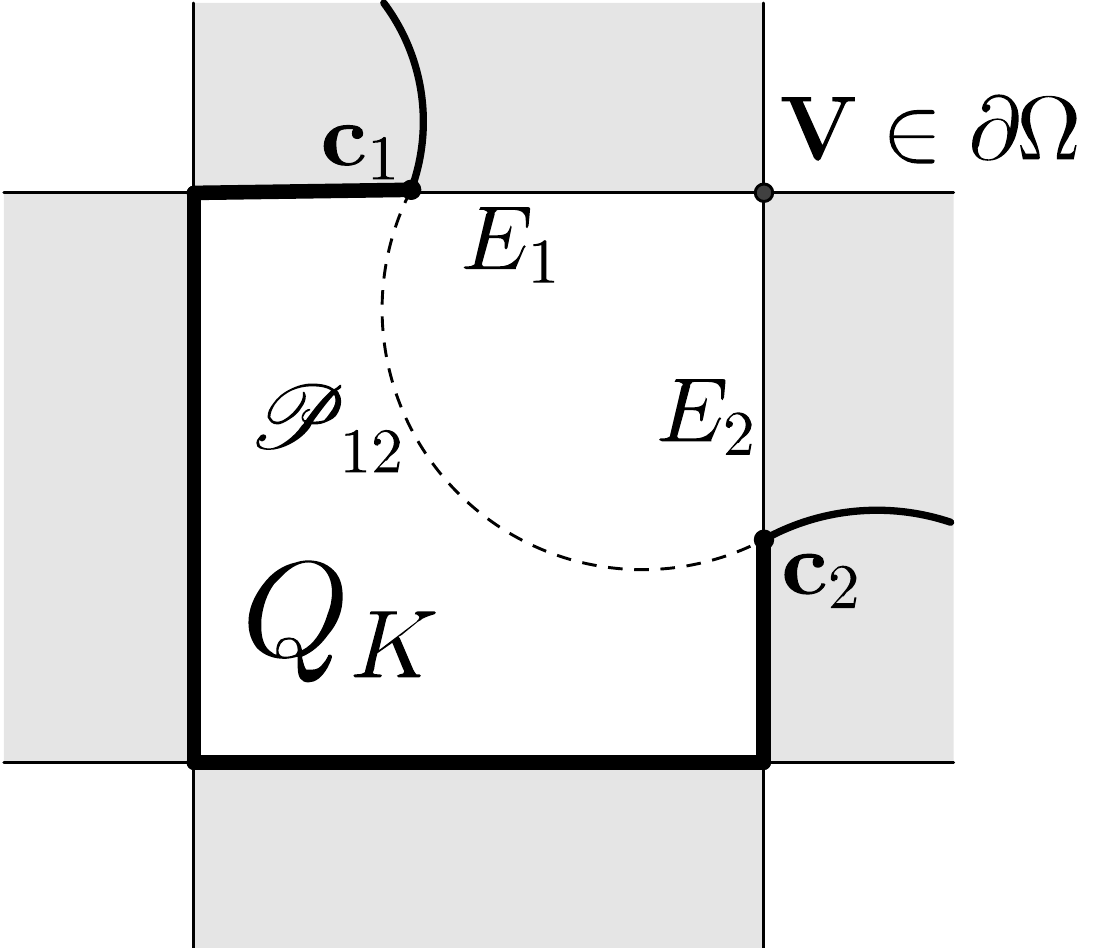}
}
\qquad
\subfigure[$\overline{E_1}\cap\overline{E_2} =\emptyset$   ]{
\includegraphics[width=0.23\textwidth]{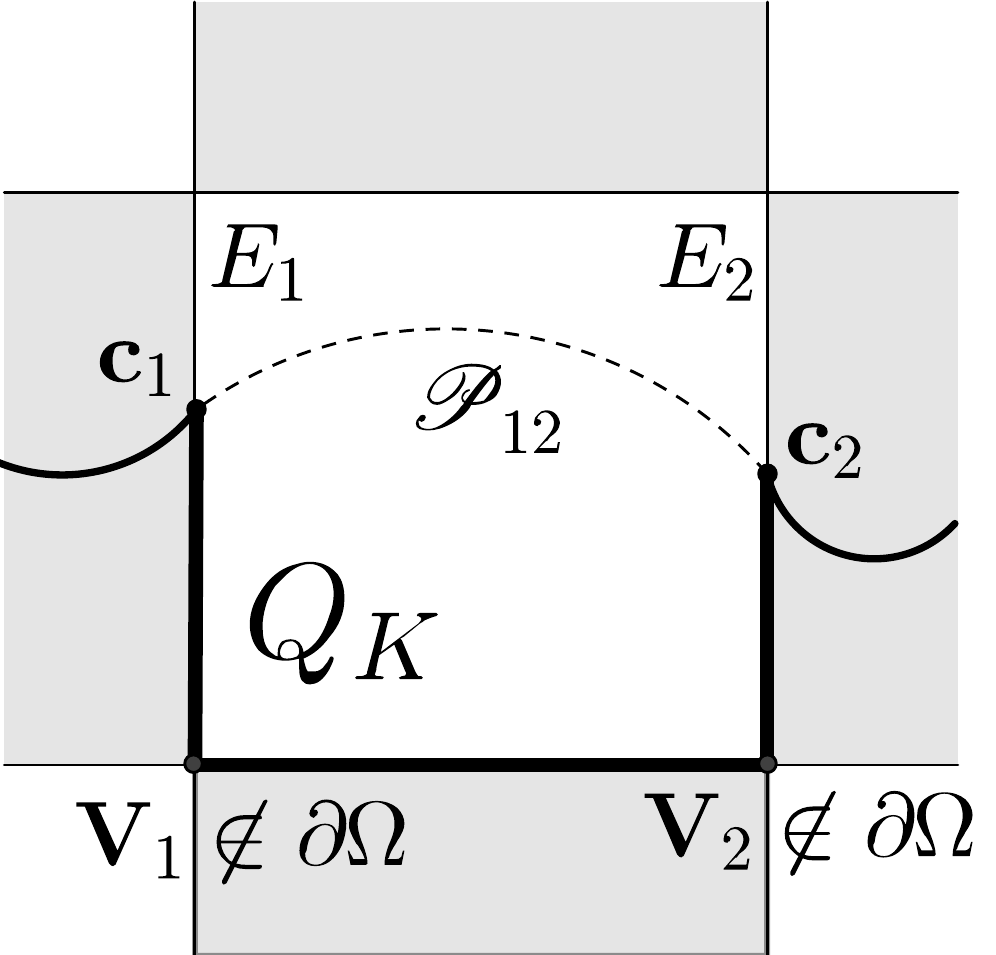}
}
\caption{The path $\cc_{12}$ in $Q_K$ (dashed line) is repaired into 
a path in  $\Ro$ (bold line)
}\label{fig:chgpath}
\end{figure} 
\end{proof}

Let $ \Mh$ be a subspace of $\mathcal{P}_{0,h}(\O)$ 
of piecewise constant functions such that
\begin{equation*}\label{def:phtil}
 \Mh \beq \{ q_h\in \mathcal{P}_{0,h}(\O)\ :\ \int_{\Ro} q_h\ d\sigma=
\int_{\Ko} q_h\ d\sigma =0 \}.
\end{equation*}
We note $V_h$ is the kernel of the operator 
$\divh: [\NChz]^2 \longrightarrow \mathcal{P}_{0,h}(\O)$.
The range of $\mathrm{div}_h$ is exactly $\Mh$ in the following lemma.
\begin{lemma}\label{lem:rangedivh}
%The range of  $\mathrm{div}_h$ is $\Mh$. That is,
\begin{equation*}\label{eq:divhncheql}
 \divh \big(  [\NChz]^2 \big) = \Mh.
\end{equation*}
\end{lemma}
\begin{proof}
From \eqref{eq:spannchz}, $[\NChz]^2$ is spanned by $\psi^{\V}[1,0],\psi^{\V}[0,1]$ 
for all interior vertices $\V$ in $\Th$.
By \eqref{eq:divvab}, we have 
\[ \divh \psi^{\V}[1,0],\ \divh \psi^{\V}[0,1] \in \Mh.  \]
It means that
\begin{equation}\label{eq:divhnchsubset}
 \divh \big(  [\NChz]^2 \big) \subset \Mh.
\end{equation}

\vspace{2mm}
If two red squares $Q_a, Q_b$ meet at an interior vertex $\V$, by  \eqref{eq:divvab},
there is a function  $\tele\in[\NChz]^2$ which is
one of $ (h/2)\psi^{\V}[\pm 1,\pm 1]$ such that
\begin{equation}\label{eq:examfqa}
 \divh \tele (Q_a)=1,\quad  \divh \tele (Q_{b})=-1,\quad  
\divh \tele(Q) =0 \mbox{ for all other } Q\in\Th.
\end{equation}

Let's fix one red  square $Q_R$ in $\Rh$.
If $Q$ is a red square in $\Rh$ different to $Q_R$, by Lemma \ref{lem:connecRB}, 
there is a path $\cc$  in $\Ro$  joining two center points of $Q$ and $Q_R$.
Let $\cc$ pass through a sequence of  
$N$ red squares $\{Q_{i}\}_{i=1}^N$ in order such that
\[ Q_1=Q_{R},\quad Q_{N}= Q, \quad {Q}_i\neq {Q}_{i+1} 
\mbox{ for } i=1,2,\cdots,N-1. \]
For each $i=1,2,\cdots,N-1$,
 $\overline{Q}_i\cap \overline{Q}_{i+1}$ is an interior vertex,
 since $\cc$ should pass there and it dose not meet $\p\O$.
Thus, as in \eqref{eq:examfqa}, there is a function $\mathbf{f}_i\in[\NChz]^2$ such that
\begin{equation*}
 \divh \tele_i (Q_i)=1,\quad  \divh \tele_i (Q_{i+1})=-1,\quad  
\divh \tele_i(Q) =0 \mbox{ if } Q\neq Q_i, Q_{i+1}.
\end{equation*}
Then, setting $\bw_h=\sum_{i=1}^{N-1} \tele_i \in\NChzvec$, we have 
\begin{equation}\label{eq:divfqr}
 \divh \bw_h (Q_1)=1,\quad  \divh \bw_h (Q_{N})=-1,\quad  
\divh \bw_h(Q) =0 \mbox{ if } Q\neq Q_1, Q_{N}. 
\end{equation}

Since these arguments can be repeated for the black squares in $\Kh$, 
\eqref{eq:divfqr} means  
the range of $\divh$ 
has at least $\N(Q)-2$ linear independent 
piecewise constant functions. 
It is combined with \eqref{eq:divhnchsubset} to complete the proof.
\end{proof}

Now, we reach at the theorem for the dimension of $V_h$.
\begin{theorem}\label{thm:dimVh}
The dimension of $V_h$ is the number of interior squares in $\Th$.
\end{theorem}
\begin{proof}
Since $V_h$ is the kernel of $\divh$, by \eqref{eq:spannchz} 
and Lemma \ref{lem:euler} and \ref{lem:rangedivh}, we have
\begin{eqnarray*}
 \dim(V_h) &=&\dim(\Xh)- \dim\Big(\divh \big(  [\NChz]^2 \big)\Big) \\
      &=& 2\N(\V^i) - (\N(Q)-2) = \#\Ih.
\end{eqnarray*}     
\end{proof}

\subsection{ Basis for $V_h$}
For each square $Q$ in $\Th$, denote by $\V^{rt}(Q), \V^{lt}(Q),\V^{lb}(Q), \V^{rb}(Q)$, 
respective vertices
of $Q$ in the right top, left top, left bottom, right bottom corners of $Q$
as depicted in Figure \ref{fig:nameconv}.
Let $Q^{rt}, Q^{lt}, Q^{lb}, Q^{rb}$ be squares in $\Th$ whose closures intersect with $\overline{Q}$
at only $\V^{rt}(Q)$, $\V^{lt}(Q)$, $\V^{lb}(Q)$, $\V^{rb}(Q)$, respectively.
If $Q\in\Th$ is a boundary square, some of them are empty.  
\begin{figure}[ht]
\hspace{48mm}
\includegraphics[width=0.35\textwidth]{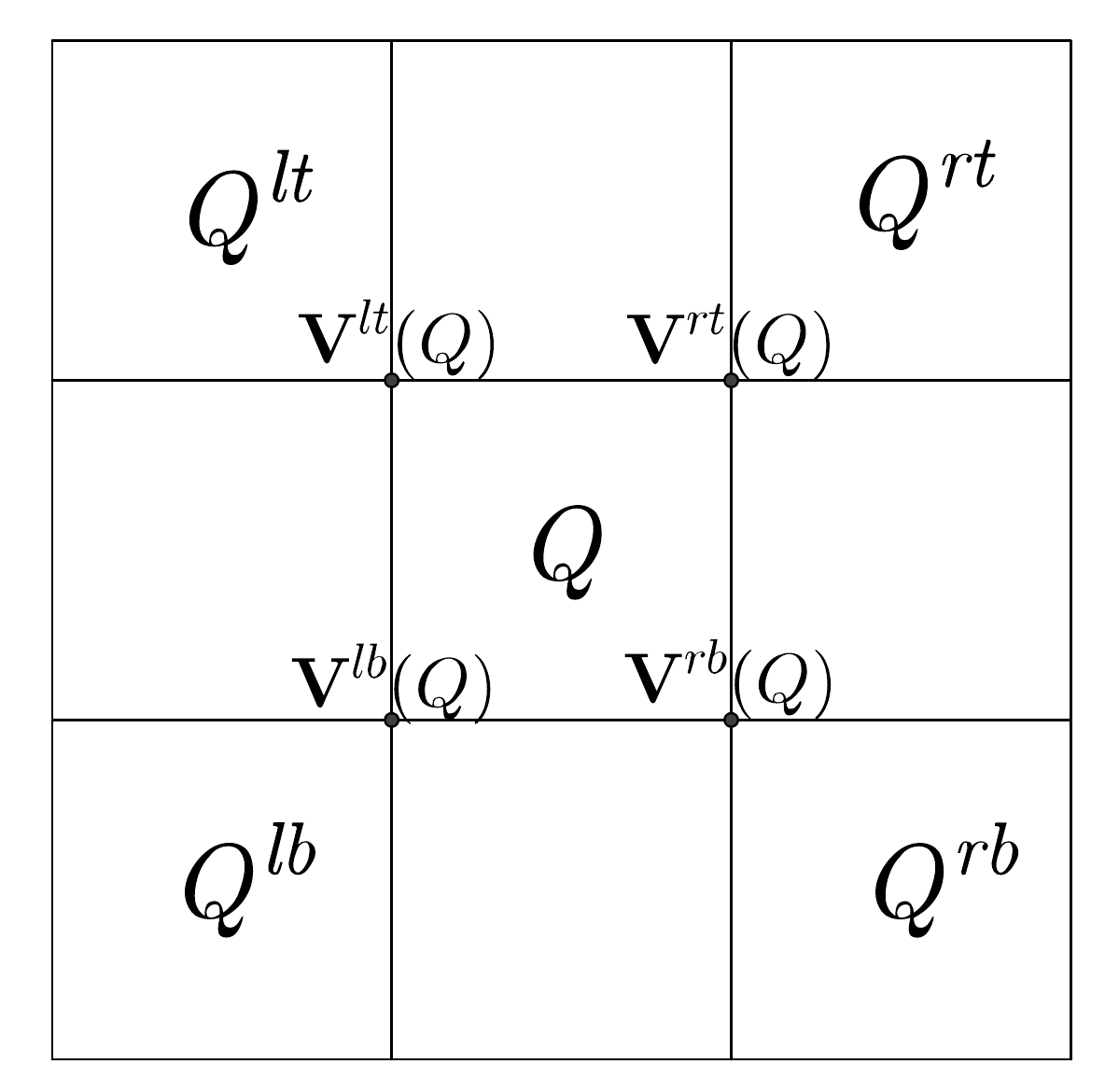}
\caption{ Name convention 
}\label{fig:nameconv}
\end{figure} 

\vspace{1mm}
For each interior square $Q\in\Ih$,
define a locally divergence-free function $\Psih^Q\in V_h$ by
\begin{equation*}
\Psih^Q =\psi^{\V^{rt}(Q)} [\frac12, -\frac12] + \psi^{\V^{lt}(Q)} [\frac12, \frac12] 
+ \psi^{\V^{lb}(Q)} [-\frac12, \frac12] +\psi^{\V^{rb}(Q)} [-\frac12, -\frac12].
\end{equation*}
The nontrivial values of $\Psih^Q$ at the 
midpoints of edges in $\Th$ are depicted
in Figure \ref{fig:div0}-(a).
With \eqref{eq:divvab}, \eqref{eq:curlvab}, 
we can easily check that $\divh \Psih^Q$ vanishes in all squares and
\begin{equation}\label{eq:curlpQ}
\curlh \Psih^Q = \left\{ \begin{array}{cl}
-4/h   & \mbox{ in } Q, \\
1/h    & \mbox{ in } Q^{rt}, Q^{lt}, Q^{lb}, Q^{rb},\\
0      & \mbox{ in other squares,}
\end{array}\right.
\end{equation}
as in  Figure \ref{fig:div0}-(b).
\begin{figure}[ht]
\subfigure[Nontrivial values of  $\Psih^Q$ at midpoints ]{
\includegraphics[width=0.45\textwidth]{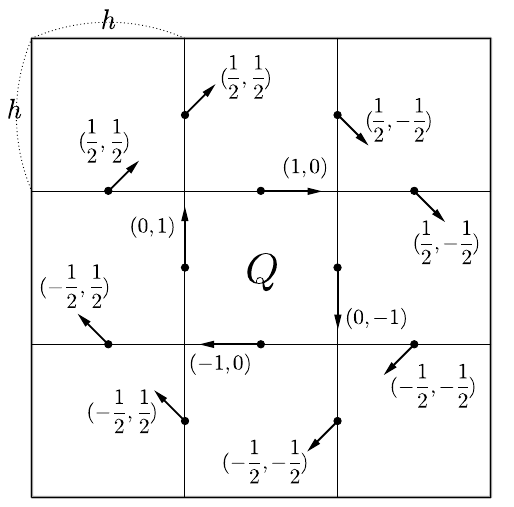}
}
\qquad
\subfigure[Nontrivial values of $\curlh\Psih^Q$ ]{
\includegraphics[width=0.45\textwidth]{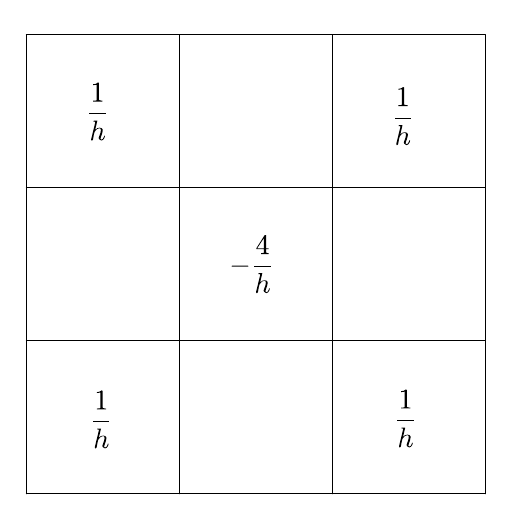}
}
\caption{ Divergence-free $\Psih^Q\in V_h$ centered at an interior square $Q$
}\label{fig:div0}
\end{figure}

Then, we are able to specify a basis for $V_h$ in the following theorem.
\begin{theorem} \label{lem:dimwh}
The set $\B = \{\Psih^{Q}\ :\  Q \mbox{ is an interior square in } \Th \}$ 
is a basis for $V_h$,
\end{theorem}
\begin{proof}
By  Theorem \ref{thm:dimVh}, the linear independency of $\B$ completes the proof.
If a linear combination of functions in $\B$ vanish,
does its discrete curl, too. Thus,  it is sufficient to prove that
the following set $\S_h$ is linearly independent, 
\[ \S_h=\{\curlh \Psih^Q \  :\  Q \mbox{ is an interior square in } \Th\}. \]

For each square $Q\in\Th$,
we define a piecewise constant function $\f_Q\in \Mhz$ by
\begin{equation}\label{eq:deffq}
\f_Q = \left\{ \begin{array}{cl}
-4   & \mbox{ in } Q, \\
1    & \mbox{ in } Q^{rt}, Q^{lt}, Q^{lb}, Q^{rb},\\
0      & \mbox{ in other squares}.
\end{array}\right.
\end{equation}
Let $N$ be the number of all red squares in $\Rh$ which are numbered as
$Q_1, Q_2, \cdots, Q_{N}$. Regarding $\f_{Q_j}$  as a column vector 
in $\R^{N}$ for each $j=1,2,\cdots,N$, we have an $N\times N$ matrix $M_R$ 
such that
\[ M_R=[\f_{Q_1}, \f_{Q_2}, \cdots, \f_{Q_{N}} ]. \]

By the definition $\f_Q$ in \eqref{eq:deffq}, all columns of $M_R$ are diagonally dominant 
and,  if $Q_j$ is a boundary square, the $j$-th column is strictly diagonally dominant.
Furthermore, by Lemma \ref{lem:connecRB}, $M_R$ is irreducible,
since $M_R(i,j)$ is nonzero 
whenever $\overline{Q_i}, \overline{Q_j}$ intersect 
for $i,j=1,2,\cdots,N$.  Thus by Taussky theorem,
$M_R$ is invertible \cite{Horn}. Then,
since $\f_Q$ is $\curlh\Psih^Q$ in \eqref{eq:curlpQ} multiplied by $h$,
the following set is linearly independent, 
\[ \S_h^R=\{\curlh \Psih^Q \  :\  Q \mbox{ is an interior red square in } \Th\}. \]

Repeating same arguments for the black squares in $\Kh$, we have 
the following $\S_h^K$ is also linearly independent,
\[ \S_h^K=\{\curlh \Psih^Q \  :\  Q \mbox{ is an interior black square in } \Th\}.\]
For a red square $Q$ and a black one $Q'$,
the supports of $\curlh \Psih^Q$ and  $\curlh \Psih^{Q'}$ do not intersect.
Therefore, we conclude that $\S_h= \S_h^R\cup  \S_h^K$ is linearly independent.
\end{proof}

The basis  for $V_h$ in Theorem  \ref{lem:dimwh}
plays
an important role in error analysis of the conforming $\Qhz-\mathcal{P}_{0,h}(\O)$
 with the dimension of $V_h$ in Theorem \ref{thm:dimVh}  \cite{Park2019}.

\section{Application for incompressible Stokes problems}\label{sec:Stokes}
Let $(\u, p)\in [H_0^1(\O)]^2\times L_0^2(\O)$ 
be the solution of the variational form
of an incompressible Stokes equation:
\begin{equation*}\label{eq:Stokes}
(\nabla \u, \nabla \v) -(p,\div \v) +(q, \div \u) = (\f, \v),
\quad\forall (\v, q)\in  [H_0^1(\O)]^2\times L_0^2(\O),
\end{equation*} 
for a source function $\f\in [L^2(\O)]^2$.

For the finite element solution, let $(\u_h, p_h)\in \Xh\times \Mh$ satisfies that
\begin{equation}\label{eq:dStokes}
(\nabla \u_h, \nabla \v_h) -(p_h,\div \v_h) +(q_h, \div \u_h) = (\f, \v_h),
\quad\forall (\v_h, q_h)\in  \Xh\times\Mh.
\end{equation}
If $\u\in [H^2(\O)]^2, p\in H^1(\O)$, the following error estimate holds
\begin{equation}\label{eq:errorbound}
|\u-\u_h|_{1,h} + \|p-p_h\|_0 \le C_{\O}h (|\u|_2 + |p|_1), 
\end{equation}
since $\Xh\times \Mh$ satisfies the inf-sup condition \cite{kim2016}.

In an alternate way to get  $(\u_h, p_h)\in \Xh\times \Mh$ in \eqref{eq:dStokes},
we can solve an elliptic problem for velocity and 
apply an explicit method for pressure as in next two subsections.

\subsection{Elliptic problem for velocity}\label{sec:velo}
The discrete velocity $\u_h$ in \eqref{eq:dStokes} satisfies that
\[ (q_h,\div \u_h)=0,\quad \mbox{for all } q_h\in\Mh. \]
Since  there exists $q_h\in\Mh$ such that $q_h=\div \u_h$ from \eqref{eq:divvab},
we have
\[ \u_h \in V_h,\]
for the locally divergence-free finite element space $V_h$ in \eqref{eq:divh0}.
Thus, to get $\u_h$ in \eqref{eq:dStokes}, we can solve the following elliptic problem
for velocity:
\begin{equation}\label{eq:femellip}
(\nabla \u_h, \nabla \v_h) = (\f, \v_h), \quad\forall \v_h\in V_h,
\end{equation}
which smaller than the problem \eqref{eq:dStokes}.

We note, since the norm is decomposed in Lemma \ref{lem:normdecom}, 
so is the inner product as
\begin{equation*}\label{eq:innprod}
 (\nabla \v_h, \nabla\bw_h)= (\divh \v_h, \divh\bw_h)    + (\curlh \v_h, \curlh\bw_h), 
\quad \forall \v_h, \bw_h\in \NChzvec.
\end{equation*}
The above means 
\begin{equation*}\label{eq:nablacurl}
 (\nabla \v_h, \nabla\bw_h)=(\curlh \v_h, \curlh\bw_h), 
\quad \forall \v_h, \bw_h\in V_h.
\end{equation*}
Thus, the problem
 \eqref{eq:femellip} is equivalent to
\begin{equation}\label{eq:femellip2}
(\curlh \u_h, \curlh \v_h) = (\f, \v_h), \quad\forall \v_h\in V_h.
\end{equation}

Suggested in Theorem \ref{lem:dimwh}, a basis $\B$ 
for $V_h$ consists of  $\Psih^{Q}$ for 
all interior square $Q\in\Th$. 
For an interior square $Q$, although the support of 
the basis function $\Psih^{Q}$ in $V_h$ consists of 9 squares,
that of  $\curlh \Psih^{Q}$ is 5 squares as in Figure \ref{fig:div0}.
If we assume $h=1$ for simplicity, as in 
Figure \ref{fig:innQQ1}, \ref{fig:innQQ3}, \ref{fig:innQQ4}, \ref{fig:innQQ5},
we have
\begin{equation*}
(\nabla \Psih^{Q} , \nabla \Psih^{Q'})=
\left\{\begin{array}{cl}
20 & \mbox{ if } Q =Q',\\
2  & \mbox{ if supports of } \nabla \Psih^{Q}, \nabla \Psih^{Q'} \mbox{ meet at 3 squares}, \\
-8  & \mbox{ if supports of } \nabla \Psih^{Q}, \nabla \Psih^{Q'} \mbox{ meet at 4 squares}, \\
1  & \mbox{ if supports of } \nabla \Psih^{Q}, \nabla \Psih^{Q'} \mbox{ meet at 1 square}.
\end{array}\right.
\end{equation*}
\begin{figure}[ht]
  \hspace{4.3cm}
\includegraphics[width=6.8cm]{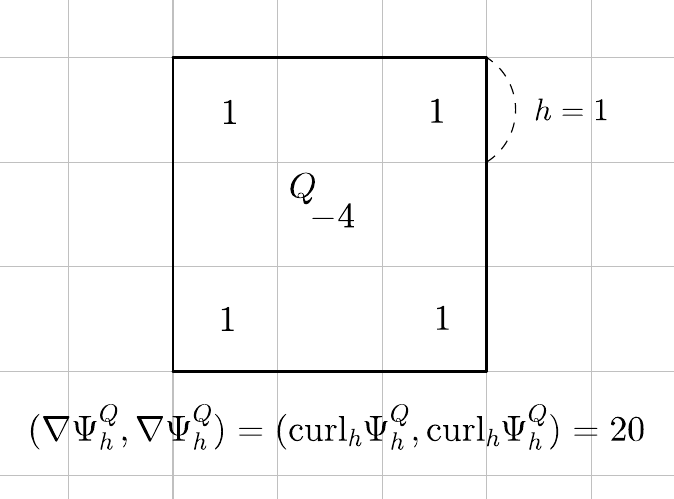}
\caption{Inner product $(\nabla\Psih^{Q},\nabla\Psih^{Q})$ when $h=1$}
\label{fig:innQQ1}
\end{figure} 
\begin{figure}[ht]
 \hspace{35.5mm}
\includegraphics[width=8cm]{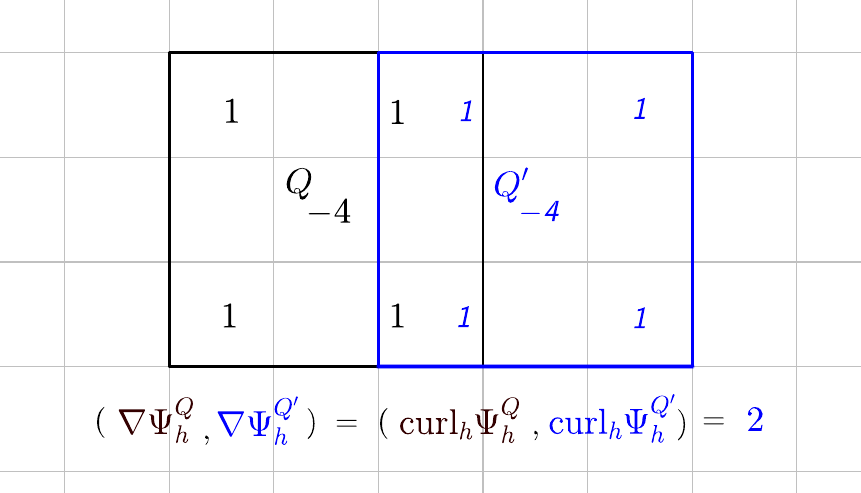}
\caption{Inner product of  $\nabla\Psih^{Q},{\color{blue}\nabla\Psih^{Q'}}$ 
when their supports meet at 3 squares}\label{fig:innQQ3}
\end{figure} 
\begin{figure}[ht]
 \hspace{35.5mm}
\includegraphics[width=8cm]{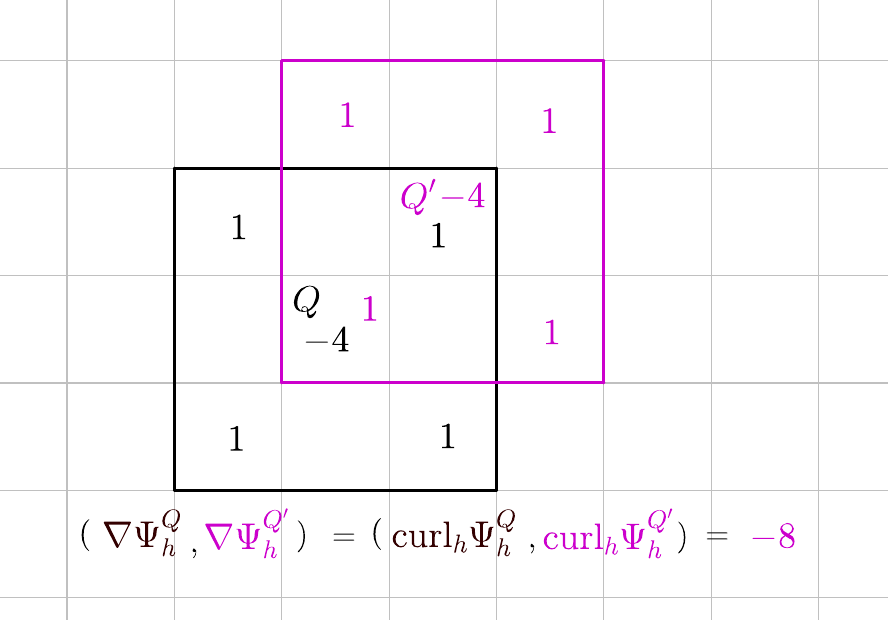}
\caption{Inner product of  $\nabla\Psih^{Q},{\color{Purple}\nabla\Psih^{Q'}}$ 
when their supports meet at 4 squares}\label{fig:innQQ4}
\end{figure} 
\definecolor{myGreen}{RGB}{0,153,102}
\begin{figure}[ht]
 \hspace{26mm}
\includegraphics[width=10.2cm]{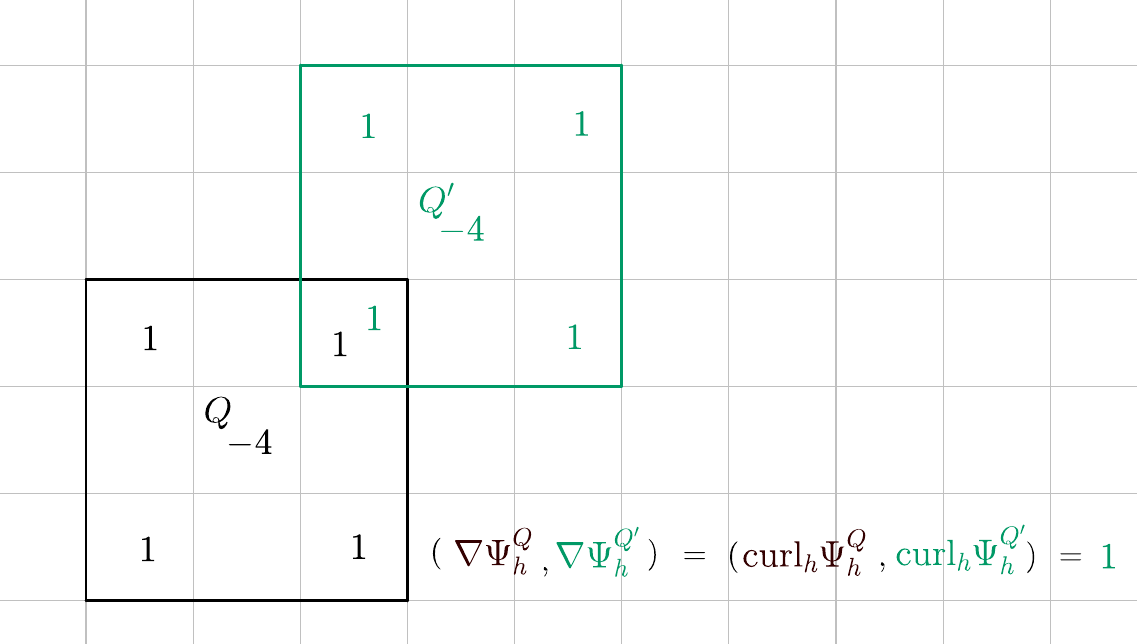}
\caption{Inner product of  $\nabla\Psih^{Q},{\color{myGreen}\nabla\Psih^{Q'}}$ 
when their supports meet at 1 square}\label{fig:innQQ5}
\end{figure} 
Besides, if $Q$ is a red square, 
the support of  $\curlh \Psih^{Q}$ lies in 5 red squares,
and vice versa for a black square. Thus, their supports do not meet each other
as in Figure \ref{fig:RKdonot}.
It means the following lemma. 
\begin{lemma}\label{lem:qrqksplit}
 For each interior red square $Q_R$ and black square $Q_K$,
$$(\nabla \Psih^{Q_R} , \nabla \Psih^{Q_K} )=0.$$
\end{lemma}

\begin{figure}[ht]
  \hspace{3.6cm}
\includegraphics[width=8.5cm]{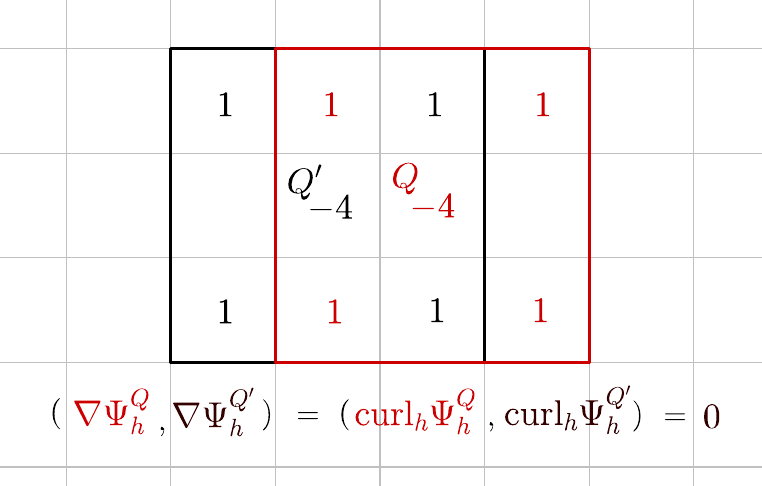}
\caption{${\color{red}\curlh\Psih^{Q}}$ for red square {$\color{red} Q$},
 $\curlh\Psih^{Q'}$ for black square  $Q'$ : They do not meet.}
\label{fig:RKdonot}
\end{figure}

As a result, when we implement the finite element method to solve \eqref{eq:femellip}
with the basis $\B$ for $V_h$,
the inner products $(\nabla \Psih^{Q} , \nabla \Psih^{Q'})$
vanish except at most 13 $Q'$  for a fixed interior square $Q\in\Th$, 
as in Figure \ref{fig:13stencil}.
 Thus
the sparsity of the system of linear equations 
is not as large as expected from the support of $\Psih^{Q}$.

\begin{figure}[ht]
\hspace{41mm}
\includegraphics[width=6.8cm]{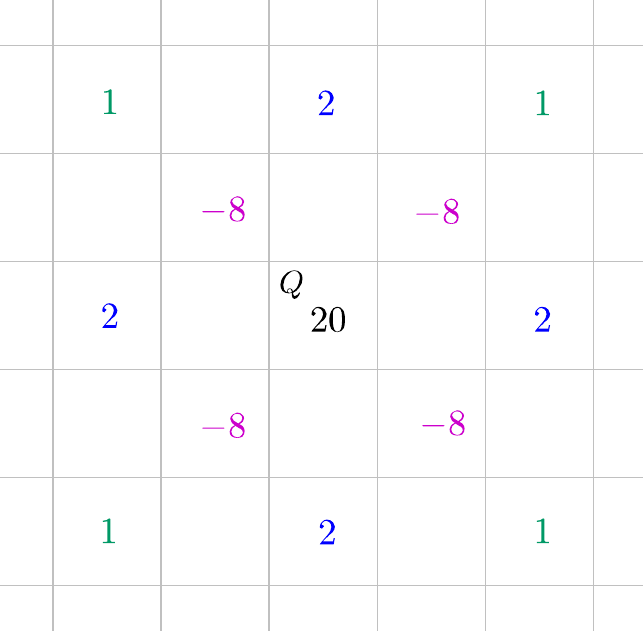}
\caption{ 13 nonzero ($\nabla\Psih^{Q},\nabla\Psih^{Q'})$ for fixed $Q$ }
\label{fig:13stencil}
\end{figure}

Furthermore, by Lemma \ref{lem:qrqksplit},
 the system of linear equations to solve \eqref{eq:femellip}
with the basis $\B$ for $V_h$
can be split into two smaller ones with the following $\B_{R}, \B_{K}$ for red and black squares,
respectively,
\begin{equation}\label{def:BrBk}
\begin{array}{ccc}
\vspace{2mm}
\B_R &=& \{\Psih^{Q}\ :\  Q \mbox{ is an interior red square in } \Th \},\\
\B_K &=& \{\Psih^{Q}\ :\  Q \mbox{ is an interior black square in } \Th \}.
\end{array}
\end{equation}

$P_1$-nonconforming finite elements  extend to 
a general polygonal domain on triangulations 
into squares and triangles \cite{Altmann}. 
Based on the proposed method, 
we can develop a method to solve velocity first 
for Stokes problems on those mixed meshes.

\subsection{Explicit method for pressure recovery}\label{sec:press}
We will find the discrete pressure   $p_h\in \Mh$ in \eqref{eq:dStokes}
by an explicit method using the a priori obtained discrete velocity 
$\u_h\in V_h$ in \eqref{eq:dStokes} 
through the elliptic problem \eqref{eq:femellip}.

Let's consider the problem (P): 
Find $a_h\in \Mh$ such that
\begin{equation*}\label{eq:phnchz}
 (a_h, \divh\vh)= (\nabla\u_h, \nabla\vh)-(\f, \vh) ,\quad \forall \vh\in \NChzvec.
\end{equation*}
By Lemma \ref{lem:rangedivh},  $p_h\in \Mh$ in \eqref{eq:dStokes}
is the unique solution of the problem (P).

Define two checkerboard functions $\chi_{R},\chi_{K} \in \Mhz$ by 
\begin{equation*}
\chi_R({Q}) = \left\{
\begin{array}{cl} 
1,\quad &\mbox{ if } Q \mbox{ is a red square,} \\
0,\quad &\mbox{ if } Q \mbox{ is a black square,} 
\end{array}\right.
\quad 
\chi_K({Q}) = \left\{
\begin{array}{cl} 
0,\quad &\mbox{ if } Q \mbox{ is a red square,} \\
1,\quad &\mbox{ if } Q \mbox{ is a black square.} 
\end{array}\right.
\end{equation*}
We have, for all $\v_h\in \Xh$,
\begin{equation}\label{eq:chirchik}
(\chi_R,\divh \v_h)=(\chi_K,\divh \v_h)=(\chi_R,\chi_K)=0.
\end{equation}

Let's fix a red  square $Q_R$, a black one $Q_K$ and define
$\hat{p_h}\in\Mhz$ by
\begin{equation*}\label{eq:defhatph}
\hat{p_h} = p_h - C_R \chi_R - C_k\chi_K,
\end{equation*}
for two constants $C_R=p_h({Q_R}), C_K=p_h({Q_K})$. 
Then, by \eqref{eq:chirchik}, $\hat{p_h}\in\Mhz$
satisfies

\begin{equation}\label{eq:phnchz2}
\begin{array}{c}
\vspace{2mm}
 (\hat{p_h}, \divh\vh)= (\nabla\u_h, \nabla\vh)-(\f, \vh) ,\quad \forall \vh\in \NChzvec,\\
\hat{p_h}({Q_R}) = \hat{p_h}({Q_K}) =0.
\end{array}
\end{equation}

We will find $\hat{p_h}$ satisfying \eqref{eq:phnchz2} by an explicit method.
If other red square $Q_R'$ meets with $Q_R$
 at an interior vertex, 
there is a function $\v_h \in[\NChz]^2$  as in \eqref{eq:examfqa} such that
\begin{equation}\label{eq:examfqa2}
 \divh \v_h({Q_R'})=1,\quad  \divh \v_h({Q_R})=-1,\quad  
\divh \v_h (Q) =0 \mbox{ for all other } Q\in\Th.
\end{equation}
Then, from \eqref{eq:phnchz2} and \eqref{eq:examfqa2},
$\hat{p_h}$ at  $Q_R'$ is determined by
\[ \int_{Q_R'} \hat{p_h}\ d\sigma =  \int_{Q_R} \hph\ d\sigma 
+   (\nabla\u_h, \nabla\vh) -(\f, \vh). \]
In this manner, 
we can find $\hph$ at all red squares by an explicit telescoping methods
using such simple $\divh\v_h$ as in \eqref{eq:examfqa2},
since all red ones are connected through interior vertices 
as in the proof of Theorem \ref{thm:dimVh}.
Applying the same telescoping methods for the black squares, 
we get a piecewise constant function $\hat{p_h}\in\Mhz$ 
satisfying \eqref{eq:phnchz2}.

Then, we can obtain the unique solution $p_h\in\Mh$ of the problem (P) so that
 $$p_h=\hat{p_h}-D_R\chi_R-D_K\chi_K,$$
where $D_R$ and $D_K$ are constants such that
\[ D_R=\frac{(\hat{p_h}, \chi_R)}{(\chi_R,\chi_R)}, \quad
 D_K=\frac{(\hat{p_h}, \chi_K)}{(\chi_K,\chi_K)}. \]

\section{Numerical results}
We chose the velocity $\u$ and pressure $p$  on $\Omega=[0,1]^2$ 
for numerical tests, as
\[\u=(\phi_y, -\phi_x),\quad  p=\sin(4\pi x) e^{\pi y},\]
where $\phi$ is the  stream function such that
\[\phi(x,y)=\sin(2\pi x)\sin(3\pi y) (x^3-x) (y^2-y). \]
The discrete velocity $\u_h$ is the sum of 
two solutions of the problems \eqref{eq:femellip2} in $V_h$
with two separable bases $\B_R$ and $\B_K$ in \eqref{def:BrBk}
for red and black squares, respectively.

The cardinalities of $\B_R$ and $\B_K$ are same and much smaller than
the dimension of the entire space $\Xh \times \Mh$ as in Table \ref{Table:dimen}. 
Each system of linear equations with $\B_R$ and $\B_K$ was solved by a direct method
based on the Cholesky decomposition.

The condition numbers of the elliptic problem \eqref{eq:femellip2} 
in $\|\cdot\|_\infty$ norm increase
with the order of $O(h^{-4})$ listed in Table  \ref{Table:cond} 
as well as those of the saddle point problem \eqref{eq:dStokes}.
The structures of non-zero entries in the concerning matrices are 
depicted in Figure \ref{fig:nzero} over $16\times 16$ mesh
with the lexicographical basis numbering.

After solving $\u_h$, we obtained the discrete pressure $p_h$
by the explicit method suggested in subsection \ref{sec:press}.
The numerical results in  Table \ref{Table:divfree} 
show the optimal order of error decay expected in \eqref{eq:errorbound}.

\begin{table}[hb]
\begin{center}
\begin{tabular}{|c|c|c|c|c|}\hline 
mesh & $\dim(\mbox{Span }\B_R) = \dim(\mbox{Span }\B_K)$ 
& $\dim\big(\Xh \times \Mh\big)$  \\\hline
8 x 8          &   18          &  160     \\\hline
16 x 16      &    98          &   704  \\\hline
32 x 32     &  450           &  2,944  \\\hline
64 x 64    &  1,922         &  12,032  \\\hline
128 x 128  & 7,938          & 48,640    \\\hline
256 x 256  & 32,258      & 195,584  \\\hline
512 x 512 &  130,050    & 784,384  \\\hline
1024 x 1024 & 522,242   & 3,141,632 \\\hline
\end{tabular}
\caption{\label{Table:dimen} Dimensions of $\mbox{Span }\B_R, \mbox{Span }\B_K$ 
and  $\Xh \times \Mh$}
\end{center}
\end{table}

\begin{table}
\begin{center}
\begin{tabular}{|c|c|c|c|c|}\hline
\multirow{3}{*}{mesh}     & 
\multicolumn{2}{c|}{ elliptic problem  \eqref{eq:femellip2}}&  
 \multicolumn{2}{c|}{saddle point problem \eqref{eq:dStokes}} \\
&\multicolumn{2}{c|}{ with  $\mbox{Span }\B_R$}&  
 \multicolumn{2}{c|}{ with   $\Xh \times \Mh$} \\
 \cline{2-5}
& condition number & order & condition number & order \\\hline
8 x 8 &8.3345E+1  &  &  1.4573E+5 &  \\\hline
16 x 16& 1.3284E+3   & 15.93  &  2.3375E+6  & 16.03  \\\hline
32 x 32&  2.1235E+4  & 15.98   &  3.7251E+7   & 15.93   \\\hline
64 x 64& 3.3968E+5  &  15.99 &  5.9424E+8   & 15.95 \\\hline
128 x 128&5.4346E+6  & 15.99  &  9.4913E+9  &  15.97 \\\hline
\end{tabular}
\caption{\label{Table:cond} condition numbers in $\|\cdot\|_\infty$ norm  }
\end{center}
\end{table}

\begin{figure}[ht]
\subfigure[Elliptic problem whose rank is $98$ ]{
\includegraphics[width=0.45\textwidth]{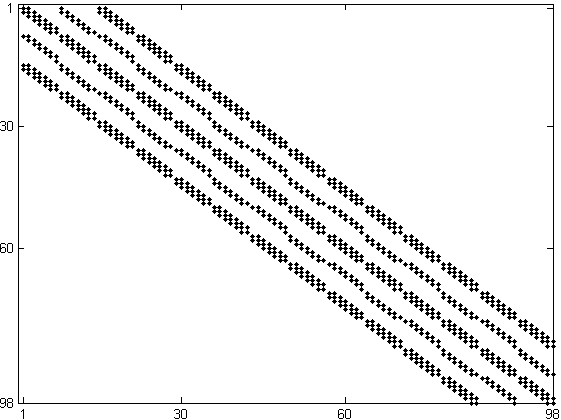}
}
\qquad
\subfigure[Saddle point problem whose rank is $704$]{
\includegraphics[width=0.45\textwidth]{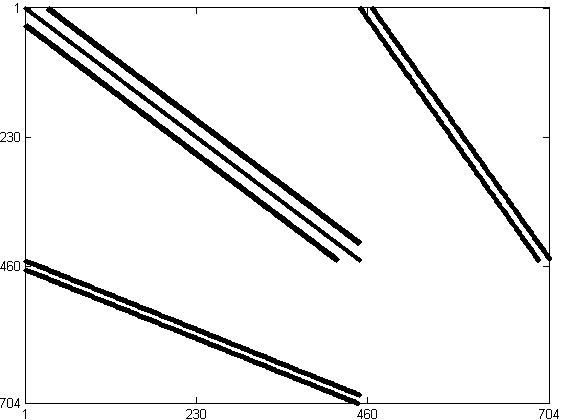}
}
\caption{Structures of non-zero entries in the concerning matrices 
over $16\times 16$ mesh
}\label{fig:nzero}
\end{figure} 

\begin{table}
\begin{center}
\begin{tabular}{|c|c|c|c|c|c|c|}\hline
\multirow{2}{*}{mesh}     & \multicolumn{2}{c|}{$\|\u -\u_h\|_{0}$}&  
 \multicolumn{2}{c|}{$|\u -\u_h|_{1,h}$}&
 \multicolumn{2}{c|}{$\|p -p_h\|_0$} \\ \cline{2-7}
& value & order & value& order&  value & order \\\hline
8 x 8 & 5.6091E-2  &  &  2.1164E+0  &  & 2.9913E+0  &  \\\hline
16 x 16& 1.3451E-2 &2.0601  & 1.0774E+0  &0.9740  & 1.5145E+0  &0.9819  \\\hline
32 x 32 &3.3342E-3  &2.0123  &  5.4115E-1  & 0.9935 &7.6081E-1  & 0.9933 \\\hline
64 x 64 &8.3187E-4  &2.0029  &  2.7088E-1  &0.9984  &3.8088E-1  & 0.9982 \\\hline
128 x 128 &2.0786E-4  &2.0007  & 1.3548E-1  &0.9996  & 1.9050E-1 & 0.9995 \\\hline
256 x 256 &5.1959E-5  & 2.0002 &6.7743E-2   &0.9999  &9.5257E-2  & 0.9999 \\\hline
512 x 512 &1.2990E-5  & 2.0000  & 3.3872E-2  &1.0000  & 4.7630E-2 & 1.0000 \\\hline
1024 x 1024 & 3.2497E-6  & 1.9990  & 1.6936E-2  &1.0000&2.3815E-2  & 1.0000 \\\hline
\end{tabular}
\caption{\label{Table:divfree} Error table for the divergence-free method}
\end{center}
\end{table}

%\section*{Acknowledgments}
%This paper was supported by Konkuk University in 2015.

\end{document}